\documentclass[12pt,reqno]{article}

\usepackage[usenames]{color}
\usepackage{amssymb}
\usepackage{graphicx}
\usepackage{amscd}

\usepackage[colorlinks=true,
linkcolor=webgreen,
filecolor=webbrown,
citecolor=webgreen]{hyperref}

\definecolor{webgreen}{rgb}{0,.5,0}
\definecolor{webbrown}{rgb}{.6,0,0}

\usepackage{amsmath}
\usepackage{amsthm}
\usepackage[nice]{nicefrac}

\newcommand{\seqnum}[1]{\href{http://oeis.org/#1}{\underline{#1}}}
\def    \C      {\mathbb{C}}

\def    \Z      {\mathbb{Z}}
\def    \sst     {\scriptscriptstyle}

\DeclareMathOperator{\diag}{diag}
\newcommand{\dsum}{\displaystyle\sum}  

\newcommand{\inv}{{^{\sst{-1}}}}
\newcommand{\Rio}{\mathfrak{R}\mathfrak{i}\mathfrak{o}}
\newcommand{\Identity}{\boldsymbol{I}}
\newcommand{\Riordan}{\boldsymbol{R}}
\newcommand{\Pascal}{\boldsymbol{P}}
\newcommand{\Aigner}{\boldsymbol{C}}
\newcommand{\Shapiro}{\boldsymbol{B}}

\begin{document}

\theoremstyle{plain}
\newtheorem{theorem}{Theorem}
\newtheorem{corollary}[theorem]{Corollary}
\newtheorem{lemma}[theorem]{Lemma}
\newtheorem{proposition}[theorem]{Proposition}

\theoremstyle{definition}
\newtheorem{definition}[theorem]{Definition}
\newtheorem{example}[theorem]{Example}
\newtheorem{conjecture}[theorem]{Conjecture}

\theoremstyle{remark}
\newtheorem{remark}[theorem]{Remark}

\begin{center}
\vskip 1cm{\LARGE\bf On One-Parameter Catalan Arrays}
\vskip 1cm
\large
Jos\'e Agapito \footnote{ Corresponding author.}, \^Angela Mestre \footnote{ Work performed within the activities of Centro de An\'alise Funcional, Estruturas Lineares e Aplica\c{c}\~oes (Faculdade de Ci\^encias, Universidade de Lisboa) and supported by the fellowship SFRH/BPD/48223/2008 provided by the Portuguese Foundation for Science and Technology (FCT).}, and Maria M. Torres\\
Centro de An\'alise Funcional, Estruturas Lineares e Aplica\c{c}\~oes\\
Grupo de Estruturas Alg\'ebricas, Lineares e Combinat\'orias\\
Departamento de Matem\'atica\\
Faculdade de Ci\^encias, Universidade de Lisboa\\
1749-016 Lisboa\\
Portugal\\
\href{mailto:jose.agapito@gmail.com}{\tt jose.agapito@gmail.com} \\
\href{mailto:amestre@fc.ul.pt}{\tt amestre@fc.ul.pt} \\
\href{mailto:mmtorres@fc.ul.pt}{\tt mmtorres@fc.ul.pt} \\
\ \\
Pasquale Petrullo\\
Dipartimento di Matematica e Informatica\\
Universit\`a degli Studi della Basilicata\\
Potenza\\
Italy\\
\href{mailto:p.petrullo@gmail.com}{\tt p.petrullo@gmail.com}
\end{center}

\begin{abstract}
We present a parametric family of Riordan arrays, which are obtained by multiplying any Riordan array with a generalized Pascal array. In particular, we focus on some interesting properties of one-parameter Catalan triangles. We obtain several combinatorial identities that involve two special Catalan matrices, the Chebyshev polynomials of the second kind, some periodic sequences, and the Fibonacci numbers.
\end{abstract}

\section{Introduction}
\label{seIntro}

There are two infinite lower triangular matrices in the mathematical literature, that are both called  Catalan triangles. Let us denote them by $\Aigner$ and $\Shapiro$, respectively. By way of illustration, their first rows are shown below:
\small
\begin{equation*}\label{eqballot}
\Aigner=\begin{pmatrix}
        1 & & & & & & &\\
        1 & 1 & & & & & &\\
        2 & 2 & 1 & & & & &\\
        5 & 5 & 3 & 1 & & & &\\
        14 & 14 & 9 & 4 & 1& & &\\
        42 & 42 & 28 & 14 & 5 & 1& &\\
        132 & 132 & 90 & 48 & 20 & 6 & 1&\\
        \vdots & \vdots & \vdots & \vdots & \vdots &\vdots &\vdots &\ddots
              \end{pmatrix} \phantom{A}\text{,}\phantom{A} 
\Shapiro=\begin{pmatrix}
        1 & & & & & & &\\
        2 & 1 & & & & & &\\
        5 & 4 & 1 & & & & &\\
        14 & 14 & 6 & 1 & & & &\\
        42 & 48 & 27 & 8 & 1& & &\\
        132 & 165 & 110 & 44 & 10 & 1 & &\\
        429 & 572 & 429 & 208 & 65 & 12 & 1 &\\
        \vdots & \vdots & \vdots & \vdots & \vdots & \vdots & \vdots & \ddots
\end{pmatrix}.
\end{equation*}
\normalsize

Note that the $0^\mathrm{th}$ column of  $\Aigner$ is given by the sequence of Catalan numbers $(C_n)_{n\ge 0}$, where $C_n=\frac{1}{n+1}{{2n}\choose{n}}$. In turn, the $0^\mathrm{th}$ column of $\Shapiro$ is given by $(C_{n+1})_{n\ge 0}$. This similarity explains why they are called Catalan triangles.

According to the online encyclopedia of integer sequences (OEIS) \cite{OEIS}, the non-zero entries of $\Aigner$ (sequence \seqnum{A033184}) are given by the formula $\Aigner_{n,k}=\frac{k+1}{n+1} {{2n-k}\choose{n}}$, for $n\ge k\ge 0$. They are known as the (ordinary) ballot numbers. Aigner \cite{Aig2008} used these numbers (and generalizations of them) to enumerate various combinatorial instances. Furthermore, Ferrari and Pinzani \cite{FPCatalan} gave an interpretation of $\Aigner$, using the ECO method, and a suitable change of basis in the vector space of one-variable polynomials. On the other hand, the non-zero entries of $\Shapiro$ (\seqnum{A039598}) are given by the formula $\Shapiro_{n,k}=\frac{k+1}{n+1} {{2n+2}\choose{n-k}}$, for $n\ge k\ge 0$. The numbers $\Shapiro_{n,k}$ first appeared in the work of Shapiro \cite{ShaCatalan,ShaRuns}, in problems connected with non-intersecting paths, random walks, Eulerian numbers, runs, slides, and moments.\footnote{Actually, Shapiro used the formula $\Shapiro_{n,k}=\frac{k}{n}{{2n}\choose{n+k}}$, for $n,k\ge 1$. His formula is just a shifted version of the formula shown previously for $\Shapiro_{n,k}$.} We will refer from now on to $\Aigner$ as the Aigner array, and to $\Shapiro$ as the Shapiro array.

The aim of this paper is to discuss some interesting features regarding the product matrix  $\Riordan( r )=\Riordan\Pascal( r )$, where $\Riordan$ is a given Riordan array, and the parameter $r$ is any real or complex number. The notation $\Pascal( r )$ stands for a generalized Pascal array, such that $\Pascal(0)=\Identity$ is the identity matrix, and $\Pascal(1)=\Pascal$ is the classical Pascal array of binomial numbers, that comes arranged as a lower triangular matrix.
We elaborate, in particular, on the product matrix  $\Aigner( r )=\Aigner\Pascal( r )$. Clearly, we have $\Aigner(0)=\Aigner$. It turns out that $\Aigner(1)=\Shapiro$. The factorization $\Shapiro=\Aigner\Pascal$ resembles Barry's definition \cite[see Sections 6.3 and 6.5]{BarryThesis} of the generalized ballot array given by $\mathbf{Bal}=\mathbf{Cat}\cdot\mathbf{Bin}$; where $\mathbf{Bin}=\Pascal$, and $\mathbf{Cat}=\diag(1,\Aigner)$.  

Recently, Yang \cite{Yang} introduced a generalized Catalan matrix, given by 
\begin{equation*}
C[a,b;r]=\left(\Big(\frac{1-\sqrt{1-4rz}}{2rz}\Big)^a,z\Big(\frac{1-\sqrt{1-4rz}}{2rz}\Big)^b\right),
\end{equation*} 
where $a,b$ are integer numbers, and $r$ is an arbitrary parameter. Such a generalization provides a unified way of presenting the matrices $\Aigner$ and $\Shapiro$, along with many other matrices related to Catalan numbers. One can check that the $0^{\mathrm{th}}$ column of $C[a,b;r]$ forms a sequence of degree $n$ monomials in $r$, whose coefficients are polynomials in $a$ and $b$, that have rational coefficients. The one-parameter Catalan triangles that we will study are instead given by 
\begin{equation*}
\Aigner( r )=\left(\frac{1-2rz-\sqrt{1-4z}}{2z(1-r+r^2z)},\frac{1-2rz-\sqrt{1-4z}}{2(1-r+r^2z)}\right).
\end{equation*}
We shall show explicitly that the $0^\mathrm{th}$ column of $\Aigner( r )$ forms a sequence $(\Aigner( r )_{n,0})_{n\ge 0}$ of degree $n$ monic polynomials in $r$. These polynomials have  positive integer coefficients. In particular, the constant term of $\Aigner( r )_{n,0}$ is the Catalan number $C_n$. For this reason, we shall refer to $\Aigner( r )$ as a Catalan triangle too. 

It is also worth mentioning that He \cite{HeCatalan} introduced recently, a family of Catalan triangles which depend on two parameters. His $(c,r)$-Catalan triangles are based on the sequence characterization of Bell-type Riordan arrays. We emphasize that they are also different from the Catalan triangles $\Aigner( r )$. One way to see this difference is to compare, once again, the $0^\mathrm{th}$ column of both types of matrices. On the one hand, one can check that the $0^\mathrm{th}$ column of a $(c,r)$-Catalan triangle  forms a sequence of degree $n$ homogeneous polynomials in $c$ and $r$, whose coefficients are the well-known Narayana numbers. On the other hand, the $0^\mathrm{th}$ column of a Catalan triangle $\Aigner( r )$ forms a sequence of one-variable polynomials of degree $n$, whose coefficients are the ordinary ballot numbers. Recall that both the Narayana and the ballot numbers are known to refine the Catalan numbers.

The fundamental theorem of Riordan arrays, and the isomorphism between Riordan arrays and Sheffer polynomials (the former being the coefficient matrices for the latter), allow us to interpret $\Aigner( r )$ as a change of basis matrix in the space of polynomials in one variable. Thus, for instance, the Shapiro array $\Shapiro$ is seen as the change of basis between the Chebyshev polynomials of the second kind, and the standard  polynomial sequence of powers of one variable. In this way, by specifying values for the one-variable of the corresponding polynomials, several formulas related to periodic sequences, the sequence of natural numbers, and the sequence of Fibonacci numbers, can be stated and generalized.

\section{General setting}
\label{seGeneralsetting}

The literature on Riordan arrays is vast and constantly growing (see, for instance, \cite{AMPT,Bri,Chen,dellaRicciaRiordan,HeCatalan,HHS,LMMSidentities,LMrecurrence,MRSV,MS11,Rog,ShGeWoWoRgroup,SprugnoliRioarrayCombsum,SprugnoliRiordanarraysAbelGould,WW}). In this section we recall the fundamental theorem of Riordan arrays, and discuss some general facts regarding the generalized Sheffer polynomial sequence associated with any Riordan array. Then we  define a one-parameter family of Riordan arrays, and focus our attention on Catalan arrays in Section \ref{serCatalan}.

\subsection{Riordan Arrays}
\label{sseRiordanarrays}

Let $d(z)=d_0+d_1z+d_2z^2+\cdots$, and  let $h(z)=h_0+h_1z+h_2z^2+\cdots$ be two formal power series in $z$ with coefficients in a given integral domain $R$, with unit $1$. Assume that $d_0=d(0)\neq0$, and $h_0=h(0)=0$. A \emph{Riordan array} $\Riordan$ is an infinite lower triangular matrix, whose entries  are given by
\begin{equation}\label{eqRiodefi}
\Riordan_{n,k}=[z^n] d(z) h(z)^k.
\end{equation}
The operator $[z^n]$ acts on a formal power series $f(z)=\sum_{\sst k\ge0}f_kz^k$ by extracting its $n^{\mathrm{th}}$ coefficient; that is, we have $[z^n]f(z)=f_n$. It is customary to denote $\Riordan=(d(z),h(z))$. In addition, the set of Riordan arrays for which $h'(0)\neq 0$, and that is equipped with the multiplication given by 
\begin{equation}\label{eqRiomult}
(d_1(z),h_1(z))  (d_2(z),h_2(z)) = (d_1(z) d_2(h_1(z)),h_2(h_1(z))),
\end{equation}
forms a group. We denote this group by $\Rio$. The group identity is the usual identity matrix $\Identity=(1,z)$. Furthermore, for any $\Riordan\in\Rio$, its inverse is given by $\Riordan^{\inv}=(1/d(\bar{h}(z)),\bar{h}(z))$, where $\bar{h}(z)$ denotes the compositional inverse of $h(z)$. 

The \emph{order} of a formal power series $g(z)=\sum_{k\ge0} g_k z^k$ is the minimal index $k$ such that $g_k\neq0$. Since $R[[z]]$ denotes the ring of formal power series in $z$ with coefficients in $R$, we will write $R_r[[z]]$ to denote the set of formal power series of order $r$. Therefore, if $\Riordan=(d(z),h(z))\in\Rio$, then $d(z)\in R_0[[z]]$, and $h(z)\in R_1[[z]]$. For simplicity, and without loss of generality, we further assume that $d_0=1$, and $h_1=h'(0)=1$, so that $\Riordan_{n,n}=h_1^n=1$, for all $n\ge 0$.

Now, let us define an action $\ast$ of the group $\Rio$ on $R_0[[z]]$ by
\begin{equation}\label{eqFTRA}
(d(z),h(z))\ast g(z) = d(z) g(h(z)).
\end{equation}
Formula \eqref{eqFTRA} represents the traditional matrix-column multiplication. Therefore, it is also  equivalent to  saying that, for any $n\ge 0$, the following identity holds
\begin{equation}\label{eqFTRA2}
\dsum_{k=0}^n \Riordan_{n,k} g_k = [z^n] d(z) g(h(z)).
\end{equation}
Equation \eqref{eqFTRA2} is known as the fundamental theorem of Riordan arrays. It is usually used to deal with sums and combinatorial identities. In particular, let $g(z)=1+xz+x^2z^2+\cdots=\frac{1}{1-xz}$. We write
\begin{equation}\label{eqzngf}
p^{\Riordan}_n(x)=\sum_{k=0}^n \Riordan_{n,k} x^k = [z^n]\dfrac{d(z)}{1-xh(z)} .
\end{equation}
We call $(p^{\Riordan}_n(x))_{n\ge 0}$ the \emph{Sheffer polynomial sequence} associated with $\Riordan$. Formula \eqref{eqzngf} implies that the generating function of $p^{\Riordan}_n(x)$ is given by
\begin{equation}\label{eqGfpnx}
\dsum_{n\ge 0} p^{\Riordan}_n(x) z^n = \dfrac{d(z)}{1-xh(z)}.
\end{equation}
Likewise, the generating function for the Sheffer polynomial associated with the inverse array $\Riordan^\inv=(1/d(\bar{h}(z)),\bar{h}(z))$ is
\begin{equation}\label{eqGfpnxinv}
\dsum_{n\ge 0} p^{\Riordan^{\inv}}_n\!\!(x) z^n = \dfrac{1}{d(\bar{h}(z))}\dfrac{1}{1-x\bar{h}(z)}.
\end{equation}

Since $\Riordan$ is invertible, and the sequence $(x^n)_{n\ge 0}$ is a basis for the linear space $\C[x]$, the sequence $(p^{\Riordan}_n(x))_{n\ge 0}$ is also a basis for $\C[x]$. Therefore, any $\Riordan\in\Rio$ can be interpreted as the change of basis matrix between the standard basis $(x^n)_{n\ge 0}$, and the basis of polynomials $(p^{\Riordan}_n(x))_{n\ge 0}$. Hence, we have
\begin{equation}\label{eqFTRApninv}
\dsum_{k=0}^n \Riordan_{n,k} p^{\Riordan^\inv}_k(x) = x^n.
\end{equation}

We will use Equation \eqref{eqFTRApninv} in Section \ref{seIdentities} to deduce some interesting combinatorial identities involving the Riordan arrays $\Aigner$ and $\Shapiro$.

\subsection{One-parameter Riordan arrays}
\label{sseParametricRiordan}

The classical Pascal array $\Pascal$ of binomial numbers ${{n}\choose{k}}$ is a Riordan array, since we can write $\Pascal=(p(z),zp(z))$, where $p(z)=\frac{1}{1-z}$. Given any real number $r$, the Riordan array $\Pascal( r )=\big(p(rz),zp(rz)\big) = \big(\frac{1}{1-rz},\frac{z}{1-rz}\big)$ is known as the \emph{generalized Pascal array of parameter} $r$. Its entries are given by $\Pascal( r )_{n,k}={{n}\choose{k}}r^{n-k}$. 

\begin{definition}\label{de1parameterRiordan}
Let $\Riordan$ be in $\Rio$. The \emph{$r$-Riordan array} $\Riordan( r )$ is defined by $\Riordan( r ) = \Riordan\Pascal( r )$. 
\end{definition}
Set $\Riordan=(d(z),h(z))$. Then, by Equation \eqref{eqRiomult}, we have $\Riordan( r ) = \left(\frac{d(z)}{1-rh(z)},\frac{h(z)}{1-rh(z)}\right)$. Using Equation \eqref{eqRiodefi}, it is easy to check that the entries of $\Riordan( r )$ are given by
\begin{equation*}
\Riordan( r )_{n,k} = \big[z^n\big] \dfrac{d(z)}{1-rh(z)}\left(\dfrac{h(z)}{1-rh(z)}\right)^k = \big[z^{n}\big] \dfrac{d(z) h(z)^k}{\big(1-rh(z)\big)^{k+1}}.
\end{equation*}

\begin{lemma} For any real number $r$, the Sheffer polynomial sequence $(p_n^{\Riordan( r )}(x))_{n\ge 0}$ satisfies the identity
\begin{equation*}
p_n^{\Riordan( r )}(x) = p_n^{\Riordan}(x+r).
\end{equation*}
\end{lemma}

\begin{proof} By Equation \eqref{eqzngf}, we have
\begin{equation*}
p_n^{\Riordan( r )}(x) = \dsum_{k=0}^n \Riordan( r )_{n,k} x^k = \big[z^n\big] \dfrac{\dfrac{d(z)}{1-rh(z)}}{1-x\dfrac{h(z)}{1-rh(z)}} = \big[z^n\big] \dfrac{d(z)}{1-(x+r)h(z)} = p^{\Riordan}_n(x+r).
\end{equation*}
\end{proof}

\begin{corollary} Let $n$ be a non-negative integer. The constant term and the sum of coefficients of $p_n^{\Riordan( r )}(x)$ are given by
\begin{align}
p_n^{\Riordan( r )}(0) &= \Riordan( r )_{n,0} = \dsum_{k=0}^n \Riordan_{n,k} r^k. \label{eqColumn0Rr} \\[1em]
p_n^{\Riordan( r )}(1) &= \dsum_{k=0}^n \Riordan( r )_{n,k} = \Riordan(r+1)_{n,0}. \label{eqRowsumsRr}
\end{align}
\end{corollary}

\section{\emph{r}-Catalan triangles}
\label{serCatalan}

Let us recall that the generating function for the well-known Catalan numbers $C_n=\frac{1}{n+1}{{2n}\choose{n}}$ is $c(z)=\frac{1-\sqrt{1-4z}}{2z}$. We apply Definition \ref{de1parameterRiordan} to  the Riordan array $\Aigner=\big(c(z),zc(z)\big)$.
\begin{definition}\label{derCatalan}
The \emph{$r$-Catalan triangle} $\Aigner( r )$ is given by $\Aigner( r ) = \Aigner\Pascal( r )$. Then\begin{equation}\label{eqrCatalan}
\Aigner( r ) = \left(\dfrac{c(z)}{1-rzc(z)},\dfrac{zc(z)}{1-rzc(z)}\right) = \left(\dfrac{1-2rz-\sqrt{1-4z}}{2z(1-r+r^2z)},\dfrac{1-2rz-\sqrt{1-4z}}{2(1-r+r^2z)}\right).
\end{equation}
\end{definition}

In particular, we have $\Aigner(0)=\Aigner$. Moreover, since the identity $c(z)=1+zc(z)^2$ holds, it follows that $\Aigner(1)=\Aigner\Pascal(1)=\big(\frac{c(z)}{1-zc(z)},\frac{zc(z)}{1-zc(z)}\big)=\big(c(z)^2,zc(z)^2\big)=\Shapiro$. 

The next formula follows from Definition \ref{derCatalan}.

\begin{lemma}\label{leCr}
The entries of $\Aigner( r )$ are given by
\begin{equation}\label{eqrCatalanentriesRiodef}
\Aigner( r )_{n,k} = \big[z^{n-k}\big] \left(\dfrac{c(z)}{1-rzc(z)}\right)^{k+1} = \dsum_{i=0}^{n-k} \frac{i+k+1}{n+1}{{2n-i-k}\choose{n}}{{i+k}\choose{k}} r^{i}.
\end{equation}
\end{lemma}

The matrix $\Aigner( r )$ is an array of polynomials in $r$. By way of illustration, the first entries of $\Aigner( r )$ are  
\small
\begin{equation*}
\Aigner( r ) =\begin{pmatrix}
      1 &  &  &  &  & \\
      1+r & 1 &  &  &  & \\
      2+2r+r^2 & 2+2r & 1 &  & & \\
      5+5r+3r^2+r^3 & 5+6r+3r^2 & 3+3r & 1 & & \\
      14+14r+9r^2+4r^3+r^4& 14+18r+12r^2+4r^3& 9+12r+6r^2& 4+4r& 1 & \\
      \vdots & \vdots & \vdots & \vdots & \vdots &  \ddots\\
    \end{pmatrix}.
\end{equation*}
\normalsize
For every $k\ge 1$, column $k$ of $\Aigner( r )$ is the $(k+1)$-fold convolution of column $0$.  Thus, for example, the first columns of $\Aigner(2)$ are given by the following sequences:
\begin{equation*}
\Aigner( 2 ) =\begin{pmatrix}
       \seqnum{A001700} &  \seqnum{A008549}  & \seqnum{A045720}  & 
\seqnum{A045894} & \seqnum{A035330} &  \cdots \\
    \end{pmatrix}.
\end{equation*}

The $0^\mathrm{th}$ column of $\Aigner( r )$ forms a sequence $(\Aigner( r )_{n,0})_{n\ge 0}$ of monic polynomials, which are given by the matrix-column product $\Aigner[1,r,r^2,\ldots]^{\sst\mathrm{T}}$, as observed in Equation \eqref{eqColumn0Rr}. We can obtain some known sequences by assigning values to the parameter $r$. For instance, we have  $(\Aigner(3)_{n,0})_{n\ge 0}=\seqnum{A049027}$ and $(\Aigner(4)_{n,0})_{n\ge 0}=\seqnum{A076025}$. Observe that, by Equation  \eqref{eqrCatalanentriesRiodef}, the constant term of each polynomial in the sequence $(\Aigner( r )_{n,0})_{n\ge 0}$ is equal to the Catalan number $C_n$. Moreover, we also have 
\small
\begin{equation*}
C_{n+1}=\frac{1}{n+2}{{2n+2}\choose{n+1}}=\frac{1}{n+1}{{2n+2}\choose{n}}=\Shapiro_{n,0}=\Aigner(1)_{n,0}=\sum_{i=0}^{n} \frac{i+1}{n+1}{{2n-i}\choose{n}}=\sum_{k=0}^n\Aigner_{n,k}.
\end{equation*}
\normalsize
This is a well-known result. In this sense, we say that the polynomial sequence $(\Aigner( r )_{n,0})_{n\ge 0}$ is a refinement of the shifted Catalan sequence $(C_{n+1})_{n\ge 0}$.

Now, set $h(z)=zc(z)$. We have $z=\bar{h}(z)c(\bar{h}(z))$. It is easy to check that $\bar{h}(z)=z-z^2$. Hence, we have $\Aigner^\inv=\big(\nicefrac{1}{c(\overline{zc(z)})},\overline{zc(z)}\big) = (1-z,z-z^2)$. Since $\Aigner( r )^\inv = \Pascal( r )^\inv\Aigner^\inv$, it follows that 
\begin{equation}\label{eqrCatalaninverse}
\Aigner( r )^\inv = \left(\dfrac{1}{1+rz},\dfrac{z}{1+rz}\right) \left(1-z,z-z^2\right) = \left(\dfrac{1+(r-1)z}{(1+rz)^2}, \dfrac{z+(r-1)z^2}{(1+rz)^2}\right).
\end{equation}

\begin{lemma}\label{leCrinv}
The entries of $\Aigner( r )^\inv$ are given by
\begin{equation}\label{eqCrinvpoly}
\Aigner( r )^\inv_{n,k}  = \big[z^{n-k}\big] \left(\dfrac{1+(r-1)z}{(1+rz)^2}\right)^{k+1} = (-1)^{n-k}\sum_{i=0}^{n-k}{{i+k+1}\choose{n-i-k}}{{i+k}\choose{k}} r^{i}.
\end{equation}
\end{lemma}

Lemma \ref{leCr} and Lemma \ref{leCrinv} show that both $\Aigner( r )_{n,k}$ and $(-1)^{n-k}\Aigner( r )^\inv_{n,k}$ are polynomials in $r$ with positive integer coefficients. In addition, note that the polynomial $\Aigner( r )_{n+1,1}$ is of degree $n$, has constant term $C_{n+1}$, and its leading coefficient is $n+1$. Both lemmas are useful to deduce several formulas and combinatorial identities. For instance, letting $r=0$ in Equation \eqref{eqCrinvpoly}, we obtain
\begin{equation*}
\Aigner^\inv_{n,k} = \big[z^{n-k}\big](1-z)^{k+1} = {{n-2k-2}\choose{n-k}}=(-1)^{n-k}{{k+1}\choose{n-k}}.
\end{equation*}
Similarly, letting $r=1$ in Equation \eqref{eqCrinvpoly}, we get 
\begin{equation}\label{eqShapiroinv}
\begin{array}{lcl}
\Shapiro_{n,k}^\inv = \big[z^{n-k}\big]\left(\frac{1}{(1+z)^2}\right)^{k+1} &=& (-1)^{n-k}{\displaystyle{{n+k+1}\choose{n-k}}} \\[2ex] &=& (-1)^{n-k}\dsum_{i=0}^{n-k}{{i+k+1}\choose{n-i-k}}{{i+k}\choose{k}}.
\end{array}
\end{equation} 
Thus, the following formula is immediate. 
\begin{lemma}
\begin{equation*}\label{eqCombident1}
{{n+k+1}\choose{n-k}}=\sum_{i=0}^{n-k}{{i+k+1}\choose{n-i-k}}{{i+k}\choose{k}}.
\end{equation*}
\end{lemma}

The next non-trivial combinatorial identity follows directly from Equation \eqref{eqRowsumsRr} and Equation \eqref{eqrCatalanentriesRiodef}.

\begin{lemma} 
\begin{equation*}\label{eqCombident2}
\dsum_{k=0}^n \dsum_{i=0}^{n-k}\frac{i+k+1}{n+1}{{2n-i-k}\choose{n}}{{i+k}\choose{k}}r^i =
\dsum_{i=0}^n\frac{i+1}{n+1}{{2n-i}\choose{n}}(r+1)^i.
\end{equation*}
\end{lemma}

As for the generating functions for the Sheffer polynomials associated with $\Aigner( r )$ and  $\Aigner( r )^\inv$, the next formulas are straightforward consequences of Equation  \eqref{eqGfpnx} and Equation \eqref{eqGfpnxinv}, using Formula \eqref{eqrCatalan} and Formula \eqref{eqrCatalaninverse}, respectively.
\begin{proposition}\label{propGfShefferCrCrinv}
\begin{align}
\dsum_{n\ge 0} p_n^{\Aigner( r )}(x) z^n &= \dfrac{c(z)}{1-(x+r)zc(z)} \notag \\[1em]
\dsum_{n\ge 0} p_n^{\Aigner( r )^\inv}\!\!(x) z^n &= \dfrac{1+(r-1)z}{1+(2r-x)z+(r^2-(r-1)x)z^2} \label{eqGfShefferCrinv}
\end{align}
\end{proposition}

In Section \ref{seIdentities},  we will use Proposition \ref{propGfShefferCrCrinv} to obtain more combinatorial identities.

\section{Special identities}
\label{seIdentities}

\subsection{Identities involving the Aigner array}

Set $r=0$ and $x=-1$ in Equation \eqref{eqGfShefferCrinv}. We obtain
\begin{align*}
1+\dsum_{n\geq 1 }p_n^{\Aigner^\inv}\!\!(-1) (-z)^n &= \dfrac{1+z}{1-z-z^2} \notag \\[2ex]
&= 1+2z+3z^2+5z^3+8z^4+13z^5+21z^6+34z^7+55z^8+\cdots \notag \\[2ex]
&= F_2 + F_3z + F_4z^2 + F_5z^3 + \cdots, 
\end{align*}
where $F_n$ is the $n^\mathrm{th}$ Fibonacci number, for $n\ge 2$. Hence, for $n\ge 0$, it holds that
\begin{equation*}
(-1)^np_n^{\Aigner^\inv}(-1)=F_{n+2}.
\end{equation*} 
Since $p_n^{\Aigner^\inv}\!\!(x)=\sum_{k=0}^n \Aigner^\inv_{n,k} x^k$, we have
\begin{equation*}
\Aigner^\inv\begin{pmatrix}1\\-1\\1\\-1\\\vdots\end{pmatrix} = \begin{pmatrix}F_2\\-F_3\\ F_4\\-F_5\\\vdots
\end{pmatrix}, \phantom{A}\text{or equivalently},\phantom{A} \Aigner\begin{pmatrix}F_2\\-F_3\\ F_4\\ -F_5\\\vdots
\end{pmatrix}=\begin{pmatrix}1\\-1\\1\\-1\\\vdots
\end{pmatrix}.
\end{equation*}
Then
\begin{equation*}
\sum_{k=0}^n {{k+1}\choose{n-k}}=F_{n+2}, \phantom{A}\text{and}\phantom{A} \sum_{k=0}^n \dfrac{k+1}{n+1} {{2n-k}\choose{n}}(-1)^{n-k} F_{k+2}=1.
\end{equation*}

\subsection{Identities involving the Shapiro array}

Set $r=1$ in Equation \eqref{eqGfShefferCrinv}. Shapiro, Woan, and Getu \cite{ShaRuns} observed that the Catalan triangle $\Shapiro$ verified the following interesting identity:
\begin{equation}\label{eqShaidentity}
\boldsymbol{B}
\left(
                             \begin{array}{c}
                               1 \\
                               2 \\
                               3 \\
                               4 \\
                               5 \\
                               \vdots
                               \end{array}
                           \right)=\left(
                             \begin{array}{c}
                               1 \\
                               4\\
                               4^2\\
                               4^3\\
                               4^4\\
                               \vdots
                             \end{array}
                           \right).
\end{equation}
The problem of finding a matrix identity that extends the sequence $(1,4,4^2,4^3,\ldots)$ to the sequence $(1,k,k^2,k^3,\ldots)$ in Equation \eqref{eqShaidentity}, was studied by Chen, Li, Shapiro and Yan \cite{Chen}. They generalized the following known recurrence relation for $\Shapiro$; namely,
$$\Shapiro_{n,k}=\boldsymbol{B}_{n-1,k-1}+2\Shapiro_{n-1,k}+\Shapiro_{n-1,k+1},$$
by constructing a matrix $\boldsymbol{M}$ that satisfies the identity
\begin{equation*}
\boldsymbol{M}
\left(
                             \begin{array}{c}
                               1 \\
                               1+t \\
                               1+t+t^2 \\
                               1+t+t^2+t^3 \\
                               \vdots
                               \end{array}
                           \right)=\left(
                             \begin{array}{c}
                               1 \\
                               k\\
                               k^2\\
                               k^3\\
                               \vdots
                             \end{array}
                           \right).
\end{equation*}
The first column of $\boldsymbol{M}$ (that is, its $0^\mathrm{th}$ column) was then interpreted in terms of weighted partial Motzkin paths. Here, we approach the problem of extending the sequence $(1,4,4^2,4^3,\ldots)$ to the sequence $(1,k,k^2,k^3,\ldots)$, by keeping $\Aigner( 1 )=\Shapiro$ fixed,  and using Sheffer sequences. More explicitly, by Equation \eqref{eqGfShefferCrinv}, we have
\begin{equation*}
1+\dsum_{n\geq 1} p_n^{\Shapiro^\inv}\!\!(x)z^n = \dfrac{1}{1-(x-2)z+z^2}.
\end{equation*}
Furthermore, by Equation \eqref{eqFTRApninv}, we also have  
\begin{equation}\label{eqShaxk}
\Shapiro\begin{pmatrix}
                               1 \\
                               p_1^{\Shapiro^\inv}\!\!(x) \\
                               p_2^{\Shapiro^\inv}\!\!(x) \\
                               \vdots \\
                             \end{pmatrix}
                           =\begin{pmatrix}
                               1 \\
                               x \\
                               x^2 \\
                               \vdots \\
                             \end{pmatrix}, \phantom{A}\text{or equivalently, }\phantom{A}
\Shapiro^\inv\begin{pmatrix}
                               1 \\
                               x \\
                               x^2 \\
                               \vdots \\
                             \end{pmatrix}
                           =\begin{pmatrix}
                               1 \\
                              p_1^{\Shapiro^\inv}\!\!(x) \\
                              p_2^{\Shapiro^\inv}\!\!(x) \\
                               \vdots \\
                             \end{pmatrix}.
\end{equation}
By Identity \eqref{eqShapiroinv}, the polynomial $p_n^{\Shapiro^\inv}\!\!(x)$ is given by
\begin{equation}\label{eqCheb}
p_n^{\Shapiro^\inv}\!\!(x) =  \sum_{k=0}^n \Shapiro^\inv_{n,k} x^k = \sum_{k=0}^n (-1)^{n-k} {{n+k+1}\choose{n-k}} x^k.
\end{equation}

Note that $p_n^{\Shapiro^\inv}\!\!(x)=U_n\left(\frac{x-2}{2}\right)$, where $(U_n(x))_{n\ge 0}$ denotes the sequence of Chebyshev polynomials of the second kind~\cite{Comtet}. Moreover, the sequence $(p_n^{\Shapiro^\inv}\!\!(x))_{n\ge 0}$ satisfies the three-term recursion
\begin{equation*}
p_n^{\Shapiro^\inv}\!\!(x)=(x-2)p_{n-1}^{\Shapiro^\inv}\!(x)-p_{n-2}^{\Shapiro^\inv}\!(x),\quad n\geq 2,
\end{equation*}
with initial values $p_0^{\Shapiro^\inv}\!\!(x)=1$ and $p_1^{\Shapiro^\inv}\!\!(x)=x-2$. Whenever $x\in\Z$, the polynomial sequence $(p_n^{\Shapiro^\inv}\!\!(x))_{n\ge 0}$ reduces to an integer sequence. By specifying the values of $x$ in Equation \eqref{eqShaxk}, we obtain many interesting relations.

\begin{example}[Periodic sequences]

When $x=2$, the recursion $p_{n}^{\Shapiro^\inv}\!\!(2)= -p_{n-2}^{\Shapiro^\inv}\!(2)$, with initial values $p_{0}^{\Shapiro^\inv}\!\!(2)=1$ and $p_{1}^{\Shapiro^\inv}\!\!(2)=0$, is solved by the period-$4$ sequence $1,0,-1,0$. Hence, we have
 \begin{equation*}
 \left(
      \begin{array}{cccccc}
        1 & & & & &\\
        2 & 1 & & & &\\
        5 & 4 & 1 & & &\\
        14 & 14 & 6 & 1 & &\\
        42 & 48 & 27 & 8 & 1&\\
        \vdots & \vdots & \vdots & \vdots & \vdots & \ddots
              \end{array}
    \right)\left(
                             \begin{array}{c}
                               1 \\
                               0 \\
                               -1 \\
                               0 \\
                               1 \\
                               \vdots
                               \end{array}
                           \right)=\left(
                             \begin{array}{c}
                               1 \\
                               2\\
                               2^2\\
                               2^3\\
                               2^4\\
                               \vdots
                             \end{array}
                           \right).
\end{equation*}
When $x=3$, the recursion $p_{n}^{\Shapiro^\inv}\!\!(3)=p_{n-1}^{\Shapiro^\inv}\!(3)-p_{n-2}^{\Shapiro^\inv}\!(3)$, satisfying the initial conditions $p_{0}^{\Shapiro^\inv}\!\!(3)=p_{1}^{\Shapiro^\inv}\!\!(3)=1$, is solved by the period-$6$ sequence $1,1,0,-1,-1,0$. Thus, we have
\begin{equation*}
\left(
      \begin{array}{cccccc}
        1 & & & & &\\
        2 & 1 & & & &\\
        5 & 4 & 1 & & &\\
        14 & 14 & 6 & 1 & &\\
        42 & 48 & 27 & 8 & 1&\\
        \vdots & \vdots & \vdots & \vdots & \vdots & \ddots
              \end{array}
    \right)\left(
                             \begin{array}{c}
                               1 \\
                               1 \\
                               0 \\
                               -1 \\
                               -1 \\
                               \vdots
                               \end{array}
                           \right)=\left(
                             \begin{array}{c}
                               1 \\
                               3\\
                               3^2\\
                               3^3\\
                               3^4\\
                               \vdots
                             \end{array}
                           \right).
\end{equation*}
These two cases are better understood by recalling that $U_n(\cos y)=\frac{\sin(n+1)y}{\sin y}$. Hence, we have $p_n^{\Shapiro^\inv}\!\!(2)=U_n\left(\cos\frac{\pi}{2}\right)$ and $p_n^{\Shapiro^\inv}\!\!(3)=U_n\left(\cos\frac{\pi}{3}\right)$. These identities explain why we obtained periodic sequences.

\end{example}

\begin{example}[Natural numbers]\label{ex:Nat}

When $x=4$, the recursion $p_n^{\Shapiro^\inv}\!(4)=2p_{n-1}^{\Shapiro^\inv}\!(4)-p_{n-2}^{\Shapiro^\inv}\!(4)$, with initial values $p_0^{\Shapiro^\inv}\!(4)=1$, $p_1^{\Shapiro^\inv}\!(4)=2$, is solved by  the sequence  $(p_n^{\Shapiro^\inv}\!(4))_{n\ge 0}=(n+1)_{n\ge 0}$. Thus, we obtain the Identity \eqref{eqShaidentity}:
\begin{equation*}
\left(
      \begin{array}{cccccc}
        1 & & & & &\\
        2 & 1 & & & &\\
        5 & 4 & 1 & & &\\
        14 & 14 & 6 & 1 & &\\
        42 & 48 & 27 & 8 & 1&\\
        \vdots & \vdots & \vdots & \vdots & \vdots & \ddots
              \end{array}
    \right)\left(
                             \begin{array}{c}
                               1 \\
                               2 \\
                               3 \\
                               4 \\
                               5 \\
                               \vdots
                               \end{array}
                           \right)=\left(
                             \begin{array}{c}
                               1 \\
                               4\\
                               4^2\\
                               4^3\\
                               4^4\\
                               \vdots
                             \end{array}
                           \right).
\end{equation*}
The previous matrix identity is equivalent to Chen's combinatorial formula \cite{Chen}:
\begin{equation*}\label{eqMotivationChen}
\sum_{k=0}^n\frac{(k+1)^2}{n+1} {{2n+2}\choose{n-k}} = 4^n.
\end{equation*}

\end{example}

\begin{example}[Fibonacci numbers]\label{ex:Fib}

When $x=5$, the recursion $p_n^{\Shapiro^\inv}\!(5)=3p_{n-1}^{\Shapiro^\inv}\!(5)-p_{n-2}^{\Shapiro^\inv}\!(5)$, with initial conditions $p_0^{\Shapiro^\inv}\!(5)=1$ and $p_1^{\Shapiro^\inv}\!(5)=3$, is solved by the subsequence of Fibonacci numbers $(p_{n}^{\Shapiro^\inv}\!(5))_{n\ge 0}=(F_{2n+2})_{n\ge 0}$. Thus, Shapiro's triangle $\Shapiro$, and the Fibonacci numbers of even index, are related by the following identity: 
\begin{equation*}
\left(
      \begin{array}{cccccc}
        1 & & & & &\\
        2 & 1 & & & &\\
        5 & 4 & 1 & & &\\
        14 & 14 & 6 & 1 & &\\
        42 & 48 & 27 & 8 & 1&\\
        \vdots & \vdots & \vdots & \vdots & \vdots & \ddots
              \end{array}
    \right)\left(
                             \begin{array}{c}
                               1 \\
                               3 \\
                               8 \\
                               21 \\
                               55 \\
                               \vdots
                               \end{array}
                           \right)=\left(
                             \begin{array}{c}
                               1 \\
                               5\\
                               5^2\\
                               5^3\\
                               5^4\\
                               \vdots
                             \end{array}
                           \right).
\end{equation*}
Equivalently, we have
\begin{equation}\label{eqFib}
\sum_{k=0}^n\frac{k+1}{n+1} {{2n+2}\choose{n-k}}F_{2k+2} = 5^n.
\end{equation}
\end{example}

Using the well-known binomial recurrence relation ${{n}\choose{k}} = {{n-1}\choose{k-1}}+{{n-1}\choose{k}}$ in Equation \eqref{eqCheb}, and then inverting Formula \eqref{eqFib}, we obtain
\begin{equation*}
F_{2n+2}=(-1)^n\sum_{k=0}^n\left({{n+k+2}\choose{n-k}}-{{n+k+1}\choose{n-k-1}}\right)(-5)^k.
\end{equation*}
Moreover, seeing as $F_{2n+3}=F_{2(n+1)+2}-F_{2n+2}$, we can deduce an analogous formula for the Fibonacci numbers of odd index; that is
\begin{equation*}
F_{2n+1}=(-1)^n\sum_{k=0}^{n}\left({{n+k+2}\choose{n-k}}-{{n+k}\choose{n-k-2}}\right)(-5)^k.
\end{equation*}

\section{Acknowledgements}
\label{se:acknow}

The authors wish to thank the editor and the referee for their helpful comments.

\bibliographystyle{jis}
\bibliography{catalan-bibarray2}

\bigskip
\hrule
\bigskip

\noindent 2010 \emph{Mathematics Subject Classification}: Primary 15B; Secondary 05A19, 11B37, 11B83, 11C.

\noindent \emph{Keywords}:  Riordan array, Pascal array, Catalan triangle, Sheffer polynomial, combinatorial identity, recurrence, Chebyshev polynomial, Fibonacci number.

\bigskip
\hrule
\bigskip

\noindent (Concerned with sequence
\seqnum{A001700},
\seqnum{A008549},
\seqnum{A033184},
\seqnum{A035330},
\seqnum{A039598},
\seqnum{A045720},
\seqnum{A045894},
\seqnum{A049027}, and
\seqnum{A076025}.)

\bigskip
\hrule
\bigskip

\end{document}